\documentclass[a4paper, 11pt, DIV=10]{scrartcl}

\title{\LARGE On the Odd Cycle Game and Connected Rules}
\author{
	Jan Corsten \and
	Adva Mond \and
	Alexey Pokrovskiy \and
	Christoph Spiegel \and
	Tibor Szab\'o}

\date{\today}

\usepackage[utf8]{inputenc}
\usepackage[T1]{fontenc}
\usepackage{lmodern}
\usepackage[british]{babel}
\usepackage{graphicx, xcolor}
\usepackage{booktabs}
\usepackage{enumerate}

\usepackage{amsmath,amsthm,amsfonts,amssymb}
\usepackage{bm}

\usepackage{microtype}

\usepackage[]{todonotes}

\usepackage{tikz}
\usetikzlibrary{calc,shapes,decorations.pathmorphing, patterns,backgrounds}

\usepackage[colorlinks=true, linkcolor=blue, filecolor=magenta, urlcolor=cyan, citecolor=gray, unicode]{hyperref}
\usepackage{cleveref}
\crefname{chapter}{Chapter}{Chapters}
\crefname{section}{Section}{Sections}
\crefname{subsection}{Section}{Sections}
\crefname{subsubsection}{Section}{Sections}
\crefname{figure}{Figure}{Figures}
\crefname{table}{Table}{Tables}
\crefname{equation}{}{}

\numberwithin{equation}{section}
\theoremstyle{definition}

\newtheorem{question}{Question}
\crefname{question}{Question}{Questions}
\newtheorem{str}{Strategy}[section]
\crefname{str}{Strategy}{Strategies}
\newtheorem{defi}{Definition}[section]
\crefname{defi}{Definition}{Definitions}

\crefname{ex}{Example}{Examples}

\theoremstyle{plain}
\newtheorem{thm}[defi]{Theorem}
\crefname{thm}{Theorem}{Theorems}
\newtheorem{conj}{Conjecture}[section]
\crefname{conj}{Conjecture}{Conjectures}
\newtheorem{lemma}[defi]{Lemma}
\crefname{lemma}{Lemma}{Lemmas}
\newtheorem{cor}[defi]{Corollary}
\crefname{cor}{Corollary}{Corollaries}
\newtheorem{claim}[defi]{Claim}
\crefname{claim}{Claim}{Claims}

\crefname{prop}{Proposition}{Propositions}
\newtheorem{obs}[defi]{Observation}
\crefname{obs}{Observation}{Observations}

\theoremstyle{remark}

\crefname{rmk}{Remark}{Remarks}

\newcommand{\R}{\mathbb{R}}

\newcommand{\N}{\mathbb{N}}

\newcommand{\e}{\varepsilon}

\newcommand{\cF}{\mathcal{F}}

\newcommand{\dist}{\mathrm{dist}}
\newcommand{\card}[1]{\left| #1 \right|}

\begin{document}
	\sloppy
	\maketitle
	
	\begin{abstract}
		We study the positional game where two players, Maker and Breaker, alternately select respectively $1$ and $b$ previously unclaimed edges of $K_n$.
		Maker wins if she succeeds in claiming all edges of some odd cycle in $K_n$ and Breaker wins otherwise.
		Improving on a result of Bednarska and Pikhurko, we show that Maker wins the odd cycle game if $b \leq \big( (4 - \sqrt6)/5 + o(1) \big) \, n$.
		We furthermore introduce ``connected rules'' and study the odd cycle game under them, both in the Maker-Breaker as well as in the Client-Waiter variant.
	\end{abstract}
	
\section{Introduction}\label{sec-intro}

Positional games are two-player combinatorial games of perfect information that are played on a finite set $X$, called the \emph{board}, equipped with a family of subsets $ \mathcal F \subset 2^X $, called the \emph{winning sets}.
Throughout this paper, the board will always be given by $E(K_n)$, the edges of the complete graph on $n$ vertices.
In our results we will deal with winning sets that are formed by the odd cycles in that complete graph.

Our main focus will be on \emph{biased Maker-Breaker games}, which were introduced by Chv\'atal and Erd\H os~\cite{ErdosChvatal} and are perhaps the most commonly studied variant of positional games.
Here the two players, referred to as Maker and Breaker, take turns selecting respectively $1$ and $b$ previously unclaimed elements of the board $X$ until all elements are claimed, with Maker starting the game.
Maker wins if she succeeds in claiming all elements of some winning set $ F \in \mathcal F $ and Breaker wins otherwise.

The value $b$ is referred to as the \emph{bias} of the game.
Most common Maker-Breaker games are an easy win for Maker whenever $b = 1$ and the board is big enough, motivating the need to study the biased version of these games where Breaker is given additional power.
We note that if Breaker has a winning strategy for some $ b \in \N$, then he also has one for any $b' \geq b$.
It follows that there exists a \emph{threshold bias} $ b_{mb}(\cF)$ so that Breaker wins the game if and only if $b \geq b_{mb}(\cF)$.
Determining that threshold bias for various natural games is one of the central problems in the study of Maker-Breaker games.

For the cycle game Bednarska and Pikhurko~\cite{BedPik2} proved that $b_{mb}(\mathcal C_n) = \lceil n/2 \rceil - 1$, where $\mathcal C_n = \{E(C) : C \text{ cycle in } K_n \}$.
Furthermore, Krivelevich~\cite{Kriv} proved that Maker can always build a linearly-long cycle when $b \leq (1/2 - o(1)) \, n$.
In~\cite{BedPik1} Bednarska and Pikhurko discussed even cycle and odd cycle games, proving that $	b_{mb}(\mathcal{OC}_n) \geq (1 - 1/\sqrt 2 - o(1)) \, n \approx 0.2928 n$ where $\mathcal{OC}_n = \{E(C) : C \text{ odd cycle in } K_n \}$.
We give the following small improvement of their lower bound.
	
\begin{thm}\label{thm-mb-oc-maker}
	The threshold bias for the Maker-Breaker odd cycle game satisfies
	\begin{equation}
		b_{mb}(\mathcal{OC}_n) \geq \left(\frac{4 - \sqrt6 }5 - o(1) \right) n \approx 0.3101 n.
	\end{equation}
\end{thm}

Since building a cycle of odd length is certainly at least as difficult for Maker as building a cycle of arbitrary length, we have the upper bound 
\begin{equation}
	b_{mb}(\mathcal{OC}_n) \leq b_{mb}(\mathcal C_n) = \lceil n/2 \rceil -1.
\end{equation}
However, no upper bound separating the two values is known, motivating Bednarska and Pikhurko to ask the following question.
	
\begin{question}[Bednarska and Pikhurko~\cite{BedPik2}] \label{qu-mb-oc}
	Do we have $b_{mb}(\mathcal{OC}_n) = \big( 1/2 - o(1) \big) \, n$?
\end{question}

We note that both in Maker's strategy presented in~\cite{BedPik1} as well as in our strategy used to prove~\cref{thm-mb-oc-maker}, Maker is maintaining a single connected component throughout the game.
In order to get closer to an answer to~\cref{qu-mb-oc}, we therefore believe it is natural to study a version of the game in which Maker is in fact forced to keep her claimed edges connected, allowing Breaker to use this information to his advantage.
We refer to these games as \emph{connected Maker-Breaker games} and denote their threshold bias by $b_{mb}^c(\cF)$.
They follow exactly the same rules as previously laid out, with the exception that Maker is now only allowed to select edges that are incident to her previously claimed edges.
If Maker is not able to make such a selection, she loses the game.

Since playing connected is a restriction for Maker, we clearly have $ b_{mb}^c(\mathcal{OC}_n) \leq b_{mb}(\mathcal{OC}_n) \leq \lceil n/2 \rceil - 1 $ for any $ n \in \N $.
We prove an upper bound that separates the bias threshold of the connected game from $n/2$, showing that the answer to~\cref{qu-mb-oc} is \emph{``No.''} if playing connected can be shown to be optimal for Maker.

\begin{thm} \label{thm-mb-oc-breaker}
The threshold bias for the connected Maker-Breaker odd cycle game satisfies
\begin{equation}
	b_{mb}^c(\mathcal{OC}_n) \leq 0.47 n.
\end{equation}
for $n$ large enough.
\end{thm}
	
Lastly, we also study the Client-Waiter\footnote{This variant of positional games was originally introduced by Beck under the name Chooser-Picker games~\cite{BeckSecondMoment,TicTacToe}. However, as this terminology frequently caused confusion, we will use the names Client and Waiter as suggested in~\cite{PickerChooser}.} version of the odd cycle game.
In every round of a \emph{biased Client-Waiter game}, Waiter offers Client $1 \leq t \leq b+1$ previously unclaimed elements of the board.
Client chooses exactly one of these elements and the remaining $t-1$ are claimed by Waiter.
Client wins if she has claimed all elements of some winning set and Waiter wins otherwise.
Here Waiter starts the game.

We note that the reason for allowing Waiter to offer less than $b+1$ elements per round is to ensure that we are again guaranteed to have a threshold bias $b_{cw}(\cF)$, so that Waiter wins the game if and only if $b \geq b_{cw}(\cF)$.
This version is commonly referred to as the \emph{monotone} version.
In the \emph{strict} version, where Waiter has to offer \emph{exactly} $b+1$ elements per round, one only has an upper and lower threshold bias.
This was not an issue in the Maker-Breaker variant where bonus moves never harm players.
For more information on bias monotonicity, we refer the reader to~\cite{Tiborbook}.
	
Hefetz, Krivelevich and Tan~\cite{WC-CW-oddC} studied the Client-Waiter cycle game and proved that $b_{cw}(\mathcal C_n) = \lceil n/2 \rceil - 1$.
Moreover, Krivelevich~\cite{Kriv} proved that Client can always build a linearly long cycle if $b \leq \big(1/2 - o(1)\big) \, n$.
For the odd cycle game, we trivially have $ b_{cw}(\mathcal{OC}_n) \leq b_{cw}(\mathcal{C}_n) = \lceil n/2 \rceil -1$.
Using a random strategy for Client, Hefetz, Krivelevich and Tan~\cite{WC-CW-oddC} proved that $b_{cw}(\mathcal{OC}_n) \geq \big( 1/(4 \log 2) - o(1) \big) \, n \approx 0.3606 n$ and conjectured that the upper bound is asymptotically tight.
	
\begin{conj}[Hefetz, Krivelevich and Tan~\cite{WC-CW-oddC}] \label{conj-cw-oc}
	We have $ b_{cw}(\mathcal{OC}_n) = \big( 1/2 - o(1) \big) \,n$.
\end{conj}
	
As with \cref{qu-mb-oc}, we will study a connected version of Client-Waiter games in order to get closer to an answer to \cref{conj-cw-oc}.
In the \emph{connected Client-Waiter game}, Waiter is only allowed to offer edges which are adjacent to some edge already claimed by Client.
If there are no such edges left, Client wins the game.
During the first round Waiter has to offer edges that are all incident to a single arbitrary vertex.\footnote{A more formal definition of these rules would consist of defining an active set of vertices at any point in the game and forcing Waiter to offer edges incident to one of these vertices. If we want Client's graph to be connected, this set of active vertices must initially consist of a single vertex, and a new vertex becomes active if and only if an ege incident to it has been claimed. Since we are playing on a complete graph, the initial vertex may be chosen arbitrarily. When playing on a non-complete graph, e.g.\ a random graph, Client would be allowed to choose this vertex at the beginning of the game.}
We denote the threshold bias of these games by $ b_{cw}^c(\cF) $, so that Waiter wins the game if and only if $b \geq b_{cw}^c(\cF) $.

Until now Client's role was most comparable to that of Maker, since both of these players are trying to claim all elements of a winning set.
However, whereas the introduction of connected rules constituted a disadvantage for Maker, they are now a restriction for Waiter, so that we have $b_{cw}^c(\cF) \geq b_{cw}(\cF)$.
Furthermore, Waiter's strategy presented in \cite{WC-CW-oddC} already follows these rules and therefore $b_{cw}^c(\mathcal C_n) = b_{cw}(\mathcal C_n) = \lceil n/2 \rceil - 1$.
Regarding the odd cycle game, we trivially have $ b_{cw}^c(\mathcal{OC}_n) \leq b_{cw}^c(\mathcal{C}_n) = \lceil n/2 \rceil -1$ and $b_{cw}^c(\mathcal{OC}_n) \geq b_{cw}(\mathcal{OC}_n) \geq \big(1/(4\log 2) - o(1) \big) \, n $.
We show that the upper bound is tight, that is we prove \cref{conj-cw-oc} under connected rules.
	
\begin{thm} \label{thm-client-occ}
	The threshold bias for the connected Client-Waiter odd cycle game satisfies
	\begin{equation}
		b_{cw}^c(\mathcal{OC}_n) = \lceil n/2 \rceil -1
	\end{equation}
	for every $n \in \N$.
\end{thm}

\medskip

We will use the following notation throughout the rest of the paper.
Given a natural number $n$ we write $[n] = \{1, \ldots, n \}$.
For a graph $G$, let $V=V(G)$ and $E=E(G)$ respectively denote its set of vertices and edges.
We write $v(G) = |V(G)|$ and $e(G) = |E(G)|$ for their cardinalities.
For a vertex $v \in V(G)$ and a set of vertices $A\subset V(G)$ let $\deg(v, A)$ denote the number of neighbours of $v$ in $A$.
Moreover, for $A, B \subset V(G)$ we use $e(A,B)$ to denote the number of edges connecting a vertex of $A$ with a vertex of $B$.
We also use $e(A)$ to denote the number of edges between vertices of $A$.
At any point during the game, we will refer to the graph given by the edges claimed by one of the players as \emph{Maker's graph}, \emph{Breaker's graph} and so forth.
These do not include isolated vertices that are not part of any edge claimed by that player.

\medskip \noindent \textbf{Outline. } We start by proving \cref{thm-mb-oc-maker} by providing a strategy for Maker in the Maker-Breaker odd cycle game in \cref{sec-mb-oc-maker}.
We then prove \cref{thm-mb-oc-breaker} by providing a strategy for Breaker under connected rules in \cref{sec-mb-oc-breaker}.
Lastly, we establish the exact bias threshold for the connected Client-Waiter odd cycle game by proving \cref{thm-client-occ} in \cref{sec-client-occ}.
Concluding remarks and open questions can be found in \cref{sec-remarks}.

\section{A strategy for Maker\texorpdfstring{ -- Proof of \cref{thm-mb-oc-maker}}{}} \label{sec-mb-oc-maker}

In this section we prove \cref{thm-mb-oc-maker} by presenting a strategy for Maker in the Maker-Breaker odd cycle game.
The basic idea will be for her to build a tree in which one side of its bipartition is as large as possible.
Maker achieves this by initially building a large star around some vertex until Breaker stops her from doing so.
She then connects an arbitrary new vertex to her tree and continues by building a star around that vertex, always using only vertices which are not already part of her tree and therefore never closing an even cycle.
She follows this principle until she is either able to close an odd cycle or she is forced to forfeit the game.

\begin{str} \label{str-mb-oc-maker}
	Throughout this strategy, let $V$ be the set of vertices in Maker's graph and $R = [n]\setminus V$.
	The strategy works in phases, starting with phase $0$.
	Let $ w_0 \in [n] $ be an arbitrary but fixed vertex.
	In every round of phase $k \geq 0$, Maker does the following.
	\begin{enumerate}[(i)] \setlength{\itemsep}{0pt}
		\item If there is an unclaimed edge closing an odd cycle, she claims it.
		\item Otherwise, if there is an unclaimed edge between $ w_k $ and $ R$, she claims it.
		\item Otherwise, if there is a vertex $ u \in R $ which is adjacent to $V \setminus \{ w_0, \dots, w_k \}$ via an unclaimed edge and the degree of that vertex in Breaker's graph is at most $|R|-b-2$, she claims this edge.
		The new vertex becomes $w_{k+1}$ and she proceeds to the next phase.
		\item Otherwise, she forfeits.
	\end{enumerate}
\end{str}

Before proving that this is a winning strategy, given a sufficiently small bias $b$, we first define a class of graphs $\mathbb G_{n,b}$ that fulfil certain properties.
In the proof of \cref{thm-mb-oc-maker} we will show that if Breaker wins the game then his final graph must belong to $\mathbb G_{n,b}$.
This will turn the question of what values of the bias \cref{str-mb-oc-maker} is successful against into the problem of minimising a certain parameter over all graphs in $\mathbb G_{n,b}$.
	
\begin{defi}\label{def:Gnq}
	We define $\mathbb G_{n,b} $ to be the set of tuples $(G, v_0, A_0, \ldots, v_s, A_s)$ satisfying the following properties.
	\begin{enumerate}[(a)]
		\item $v_0, \ldots, v_s \in [n]$ are distinct.
		\item $A_0, \ldots, A_s \subset [n] \setminus \{v_0,\ldots v_s\}$ are pairwise disjoint non-empty sets.
		\item $G$ is a graph with $V(G) \subseteq [n]$ that contains all edges which are
		\begin{itemize}
			\item[-] inside $ \{v_0, \ldots, v_s\} $ and inside $A_0 \cup \dots \cup A_s$,
			\item[-] between $\{v_0, \ldots, v_s\}$ and $ R= [n] \setminus \left(\{v_0, \ldots v_s\} \cup A_0 \cup \dots \cup A_s \right)$ and
			\item[-] between $ v_i $ and $ A_j $ for all $ 0\leq i<j\leq s $.
		\end{itemize}
		\item Every $ v \in R$ is either fully connected to $[n]\setminus R $ in $G$ or we have $ \deg(v,R) \geq |R|-b-1 $.
	\end{enumerate}
	Note that $G$ is allowed to contain more than the required edges.
	When the elements of the tuple are clear from context, we will sometimes write $G$ as an abbreviation for $(G, v_0, A_0, \ldots, v_s, A_s)$.
\end{defi}
	
\begin{proof}[Proof of \cref{thm-mb-oc-maker}]
Assume that Breaker wins the game even though Maker plays according to \cref{str-mb-oc-maker}.
Let $B_k$ denote the set of neighbours of $w_k$ claimed in phase $k$ by Maker.
Note that all edges connecting $w_k$ and $B_k$ are claimed in part (ii) of phase $k$ of the strategy.
Let $t \geq 0$ be the phase of Maker's strategy in which she either has to forfeit or the game ends naturally.
Denote Maker's and Breaker's final graphs respectively by $G_M$ and $G_B$ and note that, as previously stated, these do not include isolated vertices that are not part of any edge claimed by that player.
	
\begin{obs}\label{obs:mb-oc-maker}
	$G_M$ is a tree with bipartition parts $\{w_0, \ldots, w_t\}$ and $B_0 \cup \dots \cup B_t$.
\end{obs}
\begin{proof}\renewcommand{\qedsymbol}{$\blacksquare$}
	Since Breaker wins the game, Maker is never able to close an odd cycle.
Since she plays according to \cref{str-mb-oc-maker}, she also never closes an even cycle.
It follows that $G_M$ is a tree and by design $\{w_0, \ldots, w_t\}$ and $B_0 \cup \dots \cup B_t$ are the parts of its bipartition.
\end{proof}
		
\begin{obs}\label{obs:mb-oc-breaker}
	We have $(G_B, w_0, B_0, \ldots, w_t, B_t) \in \mathbb G_{n,b}$.
\end{obs}
\begin{proof}\renewcommand{\qedsymbol}{$\blacksquare$}
	Let us check that the requirements of \cref{def:Gnq} apply to $G_B$.
	(a) and (b) follow immediately from \cref{str-mb-oc-maker}.
	Note that the $B_i$ are non-empty since for every $1 \leq k \leq t$ the vertex $w_k$, when being chosen at the end of phase $k-1$, is adjacent to at least $b+1$ unclaimed edges going to $R$, so Maker can follow part (ii) in phase $k$ of her strategy even after Breaker has claimed his $b$ edges.
	
	Continuing onto (c), we note that Breaker must have claimed all edges inside $\{w_0, \ldots,w_t\}$ and inside $B_0 \cup \dots \cup B_s$, since otherwise Maker would have closed an odd cycle and won the game.
	Furthermore, for each $0 \leq k \leq t$, Breaker must have claimed all edges between $w_k$ and $R \cup B_{k+1} \cup \dots \cup B_t$ because otherwise Maker would have kept playing in phase $k$ for at least one more round.
	
	Finally, in order to verify (d), we will show that Maker always forfeits the game, that is the game never ends naturally.
	If Maker has not already forfeited during a previous part of the game, she is clearly forced to do so when $|R| \leq b+1$ since no vertex can have negative degree.
	With this observation, the desired property follows immediately from Maker's strategy.
\end{proof}
		 
Define the parameter
\begin{equation}
	f: \mathbb G_{n,b} \to \R, \quad (G, v_0, A_0, \ldots, v_s, A_s) \mapsto \frac{e(G)}{|A_0| + \dots + |A_s| + s}
\end{equation}
and let
\begin{equation}
	m = \min \big\{ f(G) : ( G, v_0, A_0, \ldots, v_s, A_s ) \in \mathbb G_{n,b} \big\}.
\end{equation}
Note that the minimum is attained since $\mathbb G_{n,b}$ is a finite set.

By \cref{obs:mb-oc-maker}, $G_M$ is a tree and therefore has $N = |B_0| + \dots + |B_t| + t$ edges.
It follows that Breaker claimed at most $N b$ edges during the game and hence 
\begin{equation}
	b \geq f(G_B, w_0, B_0, \ldots, w_t, B_t) \geq m.
\end{equation}

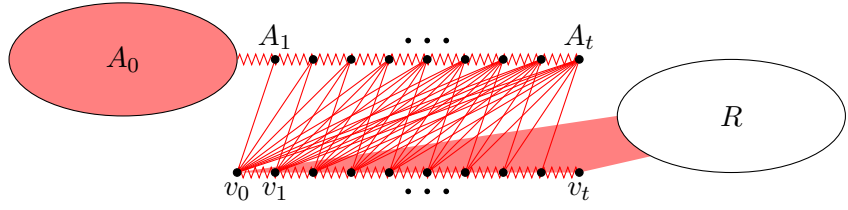
\begin{figure}
	\centering
	\begin{tikzpicture}
	\def \xrad {1.5cm}
	\def \yrad {0.5*\xrad}
	\def \h {2*\yrad}
	\def \r {0.05cm}
	\def \dist {0.5cm}
	
	\coordinate (P1) at (\xrad + 1 * \dist,\h);
	\coordinate (P2) at (\xrad + 2 * \dist,\h);
	\coordinate (P3) at (\xrad + 3 * \dist,\h);
	\coordinate (P4) at (\xrad + 4 * \dist,\h);
	\coordinate (P5) at (\xrad + 5 * \dist,\h);
	\coordinate (P6) at (\xrad + 6 * \dist,\h);
	\coordinate (P7) at (\xrad + 7 * \dist,\h);
	\coordinate (P8) at (\xrad + 8 * \dist,\h);
	\coordinate (P9) at (\xrad + 9 * \dist,\h);
	\coordinate (v0) at (\xrad,0);
	\coordinate (v1) at (\xrad + 1 * \dist,0);
	\coordinate (v2) at (\xrad + 2 * \dist,0);
	\coordinate (v3) at (\xrad + 3 * \dist,0);
	\coordinate (v4) at (\xrad + 4 * \dist,0);
	\coordinate (v5) at (\xrad + 5 * \dist,0);
	\coordinate (v6) at (\xrad + 6 * \dist,0);
	\coordinate (v7) at (\xrad + 7 * \dist,0);
	\coordinate (v8) at (\xrad + 8 * \dist,0);
	\coordinate (v9) at (\xrad + 9 * \dist,0);

	\node[draw=black,fill=red!50, ellipse, minimum width = 2 * \xrad, minimum height = 2 * \yrad] (A) at (0,\h) {$A_0$};
	\node[draw, fill=white, ellipse, minimum width = 2 * \xrad, minimum height = 2 * \yrad] (R) at (2*\xrad + 10 * \dist,\h/2) {$R$};

	\draw[red, decorate,decoration={zigzag, segment length = 3, amplitude = 2}] (A) -- (P1) -- (P2) -- (P3) -- (P4) -- (P5) -- (P6) -- (P7) -- (P8) -- (P9);
	\draw[red, decorate,decoration={zigzag, segment length = 3, amplitude = 2}] (v0) -- (v1) -- (v2) -- (v3) -- (v4) -- (v5) -- (v6) -- (v7) -- (v8) -- (v9);
	
	
	\begin{pgfonlayer}{background}
	\fill[red!50] (v0) -- (v9) -- (R.south west) -- (R.west) -- cycle;
	\end{pgfonlayer}
	
	\foreach \i in {0,1,...,8} {
		\pgfmathsetmacro\k{\i + 1}
		\foreach \j in {\k,...,9} 
		\draw[red] (v\i) -- (P\j);
	}
	
	

	
	\foreach \x in {(P1), (P2), (P3), (P4), (P5), (P6), (P7), (P8), (P9)}  
	\draw [fill = black] \x circle [radius = \r];
	
	\foreach \x in {(v0),(v1), (v2), (v3), (v4), (v5), (v6), (v7), (v8), (v9)}  
	\draw [fill = black] \x circle [radius = \r];
	
	\draw (v0) node[below]{$v_0$};
	\draw (v1) node[below]{$v_1$};
	\draw (v9) node[below]{$v_t$};
	\foreach \x in {-1,0,1}{
		\node[fill=black, circle, inner sep =0cm, outer sep =0cm, minimum size= 0.075cm] at (\xrad + 5 * \dist + \x * \dist/2 ,-0.25cm) {};
		\node[fill=black, circle, inner sep =0cm, outer sep =0cm, minimum size= 0.075cm] at (\xrad + 5 * \dist + \x * \dist/2 ,\h+0.25cm) {};
	}
	\draw (P1) node[above]{$A_1$};
	\draw (P9) node[above]{$A_t$};
	\end{tikzpicture}
	\caption{A graph $G \in \mathbb G_{n,b}$ that satisfies \cref{claim:Aj1} with all required edges in red.}
	\label{fig-mb-oc-maker}
\end{figure}

We have therefore turned the problem into a minimisation over $\mathbb G_{n,b}$. Fix some $(G,v_0,A_0, \ldots, v_s,A_s) \in \mathbb G_{n,b}$ which minimises $f$. In order to simplify this minimisation problem, we will make some observations about the structure of $G$. 

\begin{claim}\label{claim:Aj1}
	We have $|A_1| = |A_2| = \ldots = |A_s| = 1$.
%
%
\end{claim}

\begin{proof}\renewcommand{\qedsymbol}{$\blacksquare$}
Assume otherwise that $|A_j|>1 $ for some $ 1 \leq j \leq s$.
By moving one vertex $ u \in A_j $ to $ A_0 $ and deleting the edges connecting $u$ and $\{v_0, \ldots, v_{j-1}\}$, we would obtain a new graph $\tilde G$ and new sets $\tilde A_0, \ldots, \tilde A_s$ satisfying $(\tilde G, v_0, \tilde A_0, \ldots, v_s, \tilde A_s) \in \mathbb G_{n,b}$ and $f(\tilde G) < f(G)$, contradicting our choice of $G$.
\end{proof}

Using \cref{claim:Aj1}, we can lower bound the number of edges in $G$ by  
\begin{align}\label{eq:e(G)}
	e(G) & \geq \binom{|A_0|+s}{2} + \binom{s+1}{2} + (s+1)|R| + \sum_{i=1}^s i + e \Big(R, \bigcup_{i=0}^{s}A_i \Big) + e(R) \nonumber \\
	&= \left( \card{A_0}^2 +3s^2 + 2\card{A_0}s + 2 \card R s \right)/2 + e \Big(R, \bigcup_{i=0}^{s} A_i \Big) + e(R) + o(n^2).
\end{align}
Note that this holds even if $s=0$.

Let $R_1$ denote the vertices in $R$ that are fully connected to $[n] \setminus R$ and let $R_2 = R \setminus R_1$ denote those that are not. Note that every vertex in $R_2$ satisfies the degree condition in $R$, that is $ \deg(v,R) \geq |R|-b-1 $ for $v \in R_2$.

\begin{claim}\label{claim:e(R)}
We have 
\[e(R,A_0 \cup \dots \cup A_s) = e(R_1,A_0 \cup \dots \cup A_s) = |R_1| \, (|A_0| + s)\]
and
\[e(R) = e(R_2) = \max \big( 0, \lceil |R_2| \, (|R|-b-1)/2 \rceil \big).\]
Consequently, we may assume that
\begin{equation*}
	R = \left\{\begin{array}{ll}
        R_1 & \text{if } |R| > 2|A_0| +2s + b, \\
        R_2 & \text{otherwise. }
        \end{array}\right.
\end{equation*}
\end{claim}

\begin{proof}\renewcommand{\qedsymbol}{$\blacksquare$}
  The first equation follows immediately from the definition of $R_1$ and the minimality of $G$.
  Note that the Erd\H os--Gallai theorem implies that for every $m > d$ there exists an $m$-vertex graph satisfying $\deg(v) = d$ for all but at most one vertex $v$ and $deg(v) \in \{d, d + 1\}$ for that remaining vertex. The second equation now follows from the definition of $R_2$ and the minimality of $G$ as well. It is now easy to see from these equations that $f(G)$ is minimised when $R=R_2$ if $|R| \leq 2|A_0| +2s + b$ and when $R=R_1$ otherwise.
\end{proof}

Furthermore, we have the following relation between $R_2$ and $A_0$.

\begin{claim}\label{claim:A0Rs}
  If $s \geq 1$ then $|R_2| + s \leq |A_0| \leq |R_2|+s+3$.
\end{claim}

\begin{proof}\renewcommand{\qedsymbol}{$\blacksquare$}
  To see that the left-hand side of the inequality holds, assume to the contrary that $|A_0| < |R_2|+s $.
Moving $v_s$ and the vertex $u \in A_s$ (of which there is only one by \cref{claim:Aj1}) to $A_0$, removing the $2s$ edges between $\{u,v_s\}$ and $ \{v_0, \ldots, v_{s-1}\}$ as well as the $|R_2|$ edges between $v_s$ and $R_2$ and adding the $|A_0|+s$ edges between $v_s$ and $A_0 \cup \dots \cup A_s$, we obtain a new graph $\tilde{G}$ with new sets $\tilde A_0, \ldots, \tilde A_{s-1}$.
Note that $(\tilde G, v_0, \tilde A_0, \ldots, v_{s-1}, \tilde A_{s-1}) \in \mathbb G_{n,b}$ and that, since we removed $|R_2| + 2s$ edges and added $|A_0|+s$ edges and by assumption $|R_2| + 2s > |A_0| + s$, we have $f(\tilde G) < f(G)$, contradicting our choice of $G$.
If $|A_0| > |R_2| + s + 3$, then we repeat this argument by removing two vertices $v, u$ from $A_0$ and setting them to be $v_{s+1} = v$ and $A_{s+1} = \{u\}$.
It is easy to verify that this leads to a similar contradiction, proving the claim.
\end{proof}

%
%
%
%

Define
\begin{equation} \label{eq:params}
	\beta = b/n, \quad \alpha = |A_0| / n, \quad \rho = |R| / n \quad \text{and} \quad \sigma = s/n
\end{equation}
and note that $|A_0| + 2s + 1 + |R| = n$, that is 
\begin{equation} \label{eq:sum1}
	\alpha + 2\sigma + \rho + o(1) = 1.
\end{equation}
Finally, we will solve the optimisation problem through a case distinction.

\medskip \noindent \textbf{Case~1.} If $|R| \leq 2|A_0| + 2s + b$, then by \cref{claim:e(R)} we must have $R = R_2$ and $e(R,A_0 \cup \dots \cup A_s) = 0$.
%
%
We note that, if $s \geq 1$, dividing the relation given in \cref{claim:A0Rs} by $n$ gives us
\begin{equation} \label{eq:A0Rs/n}
	\alpha = \rho + \sigma + o(1).
\end{equation}

\medskip \noindent \emph{Case~1.1.} If $|R| \leq b+1$ then we have $e(R) = 0$ by \cref{claim:e(R)}.
Plugging the parameters into \cref{eq:e(G)} therefore gives
\begin{align*}
	\beta \geq \frac{f(G)}{n} & = \frac{e(G)/n^2}{(|A_0| + 2s)/n} \geq \frac{\alpha^2 +3\sigma^2 + 2\alpha\sigma + 2 \rho\sigma}{2 (\alpha + 2\sigma)} + o(1).
\end{align*}

If $s=0$ (and therefore $\sigma = 0$), we get
  \[ \beta \geq \frac{\alpha^2}{2\alpha} = \frac \alpha 2 = \frac{1-\rho}{2} \geq \frac{1 - \beta}{2}\]
and thus $\beta \geq 1/3$.

If $s \geq 1$ then using both \cref{eq:sum1} and \cref{eq:A0Rs/n} we get
\begin{align*}
	\beta \geq \frac{2-2\rho-\rho^2}{6 (1-\rho)} + o(1).
\end{align*}
By assumption we also know that $\beta \geq \rho + o(1)$, so that after minimising the maximum of $\rho$ and $(2-\rho^2-2\rho)/(6 (1-\rho))$ over $0 \leq \rho < 1$, we get that $\beta \geq (4-\sqrt{6})/5 + o(1)$.

\medskip \noindent \emph{Case~1.2.} If $|R| > b+1$, then $e(R) \geq |R| \left(|R|-b-1\right)/2$ by \cref{claim:e(R)}.
Again plugging the parameters into \cref{eq:e(G)} therefore gives
\begin{align}\label{eq:case12}
  \beta \geq \frac{\alpha^2 +3\sigma^2 + 2\alpha\sigma + 2 \rho\sigma + \rho(\rho - \beta)}{2 (\alpha + 2\sigma)} + o(1).
\end{align}

If $s=0$, we have $\alpha = 1 - \rho + o(1)$ by \cref{eq:sum1} and hence \cref{eq:case12} becomes
\[ \beta \geq \frac{\alpha^2 + \rho (\rho - \beta)}{2 \alpha} = \frac{1 + 2 \rho^2 - 2 \rho - \rho \beta}{2-2\rho}.\]
Solving for $\beta$ we get get $ \beta \geq (1-2\rho+2\rho^2)/(2-\rho)$, which implies $\beta \geq 2 - \sqrt{5/2}$.

If $s \geq 1$, combining \cref{eq:sum1,eq:A0Rs/n} implies
\begin{align*}
	\beta \geq \frac{2-2\rho+2\rho^2}{6 - 3\rho} + o(1).
\end{align*}
By assumption we also know that $\beta \leq \rho + o(1)$.
We can verify that $(2-2\rho+2\rho^2)/(6 - 3\rho) > \rho$ when $(4 - \sqrt{6})/5 < \rho \leq 1$, so that after minimising $(2-2\rho-2\rho^2)/(6 - 3\rho))$ over $0 \leq \rho \leq (4 - \sqrt{6})/5$, we again obtain that $\beta \geq (4-\sqrt{6})/5 + o(1)$.

\medskip \noindent \textbf{Case~2.} If $|R| > 2|A_0| + 2s + b$, then by \cref{claim:e(R)} we may assume that $R = R_1$.
Using that $R = R_2$ in \cref{eq:e(G)} gives us 
\begin{equation}\label{eq:case2}
	\beta \geq \frac{\alpha^2 +3\sigma^2 + 2\alpha\sigma + 2 \rho (1 - \rho) }{2 (\alpha + 2\sigma)} + o(1).
\end{equation}
If $s=0$, we again have $\alpha = 1 - \rho + o(1)$ by \cref{eq:sum1} and hence \cref{eq:case2} becomes
\begin{align*}
  \beta &\geq \frac{\alpha^2 + 2 \rho (1 - \rho) }{2 \alpha} + o(1) = \frac{(1-\rho)^2 +2\rho (1-\rho)}{2(1-\rho)} = \frac{1+\rho}{2} \geq 1/2.
\end{align*}

If $s \geq 1$ then again dividing the relation given in \cref{claim:A0Rs} by $n$ gives us
\begin{equation} \label{eq:A0s/n}
	\alpha = \sigma + o(1).
\end{equation}
Combining \cref{eq:case2,eq:A0s/n,eq:sum1} implies that $\beta \geq 1 - 2\alpha + o(1)$.
However, the case assumption combined with \cref{eq:A0s/n,eq:sum1} also gives us that $\alpha \leq (1 - \beta) / 7 + o(1)$.
Combining these two results, we get that $\beta = 1 + o(1)$.

\medskip

All cases resulted in $\beta \geq (4-\sqrt{6})/5 + o(1)$ so that $b \geq \big( (4-\sqrt{6})/5 + o(1) \big) \, n$.
\end{proof}

\section{A strategy for Breaker\texorpdfstring{ -- Proof of \cref{thm-mb-oc-breaker}}{}} \label{sec-mb-oc-breaker}

In this section we prove \cref{thm-mb-oc-breaker} by presenting a strategy for Breaker in the connected Maker-Breaker odd cycle game.
Note that, as long as Maker has not yet won the game, her graph will always be bipartite.
Besides blocking any immediate threats of Maker creating an odd cycle, Breaker's goal will be to connect the vertices not yet touched by Maker in as even a way as possible to the two parts of Maker's graph.
This way Breaker attempts to minimise the number of his edges ending up between the two parts of Maker's graph, where they would not serve the purpose of blocking any odd cycle.

Let $ G_M(s) = (V_s,E_s) $ denote Maker's graph after her $s$--th turn and let $R_s = [n] \setminus V_s$ be the set of vertices not touched by Maker.
Since $ G_M(s) $ is bipartite and since Maker is forced to play connected, there is a unique (up to labelling) bipartition $ V_s = V_s^1 \cup V_s^2 $, which we may choose in such a way that $ V_s^i \subset V_{s+1}^{i} $ holds for all $ s \geq 0 $ and $ i \in \{1,2\}$.
Note again that our notion of Maker's and Breaker's graph do not include isolated vertices.
 
A \emph{state} is a tuple $(s,k)$ with $ s \geq 1$ and $0 \leq k \leq b$ and describes the situation of the game after the $s$--th move of Maker and after Breaker has claimed $k$ edges in his $s$--th turn.
For example, the state $(1,0)$ describes the situation right after Maker claimed the first edge of the game.
We let $\deg_k(v,V_s)$, $\deg_k(v,V_s^1)$ and $\deg_k(v,V_s^2)$ denote the number of neighbours of a vertex $v \in [n]$ in $V_s$, $V_s^1$ and $V_s^2$ in Breaker's graph at state $(s,k)$.
We similarly define other quantities such as $e_k(V_s)$, $e_k(R_s)$, $e_k(R_s,V_s)$, $e_k(R_s,V_s^1)$ and $e_k(R_s,V_s^2)$.
When we refer to these quantities at the end of round $s$, that is when $k=b$, we will omit the extra parameter and simply denote them by $\deg(v, V_s)$ and so forth.

Unless stated otherwise, we are always referring to Breaker's graph whenever talking about edges, degrees, etc. for the remainder of this section.
Using this notation, let us state Breaker's strategy.
	
\begin{str} \label{str-mb-oc-con-br}
	After the $s$--th move of Maker, Breaker does the following.
	Any ties in this strategy are always broken arbitrarily unless otherwise specified.
	\begin{enumerate}[(i)]
		\item He selects every unclaimed edge inside $V_s^1$ or inside $ V_s^2 $, killing all threats of Maker completing an odd cycle with her next move.
		If this is not possible, he forfeits.
		\item Assume Breaker has fulfilled part~(i) by claiming $0 \leq k_0 < b$ edges and is in his $k$--th move for $k_0 < k \leq b$.
		Also assume that $e_k(V_s^{i_1},R_s) \leq e_k(V_s^{i_2},R_s)$ where $\{i_1,i_2\} = \{1,2\}$.
		\begin{enumerate}[(a)]
			\item If $|R_s| \leq b$ and there are at least $1$ and at most $b-k+1$ unclaimed edges between $V_s^i$ and $R_s$ for some $i \in \{1,2\}$, then he claims an arbitrary edge between $V_s^i$ and $R_s$.
			\item Otherwise, if there are unclaimed edges between $R_s$ and $V_s^{i_1}$, he selects an edge that connects $V_s^{i_1}$ to some vertex $ v \in R_s $ minimising $\deg_k(v,V_s^{i_1})$.
        	If there are multiple options, he takes one for which $\deg_k(v,V_s^{i_2}) - \deg_k(v,V_s^{i_1})$ is maximised.
			\item Otherwise, if there are unclaimed edges between $R_s$ and $V_s^{i_2}$, he selects an edge that connects $V_s^{i_2}$ to some vertex $ v \in R_s $ minimising $\deg_k(v,V_s^{i_2})$.
			\item Otherwise he claims an arbitrary edge.
		\end{enumerate}
	\end{enumerate}
\end{str}

From now on we assume that Breaker is given a bias of
\begin{equation}
	b = \left\lceil \frac{n-\e n}{2} \right\rceil,
\end{equation}
where $\e = 0.06$. Let us furthermore assume that, despite Breaker following \cref{str-mb-oc-con-br} with the given bias, Maker wins the game in the $(t+2)$--nd round for some $t \geq 0$.
We begin with some easy facts.

\begin{obs} \label{Obs1}
	We may assume the following throughout the game.
	\begin{enumerate}[(i)] \setlength{\itemsep}{0pt}
		\item Maker's graph $G_M(s)$ is a tree for all $ 1 \leq s \leq t + 1$ and in particular $|V_s| = s+1$.
		\item Maker must create at least $b+1$ threats in round $t+1$.
		\item We have $t \leq n - 2$.
	\end{enumerate}
\end{obs}

\begin{proof}
	Let us start by proving (i).
	Suppose that at some point Maker claims an edge $e$ which closes an even cycle $C_1$ in her graph.
	Let $C_2$ be the odd cycle that she closes in the $t+2$--nd round.
	If $e \in C_2 $, then $(C_2 \setminus C_1) \cup (C_1 \setminus C_2) $ is another odd cycle not containing $e$.
	It follows that we may assume that she chose any other edge instead of $e$ without decreasing her possibility of winning.
	If at some point there are only edges left which close an even cycle in Maker's graph, she will lose the game, contradicting our assumption.
	
	For (ii), note that the only way Breaker loses while following \cref{str-mb-oc-con-br} is when he cannot defend all threats and is forced to forfeit, implying that he must have faced $b+1$ threats that were newly created by Breaker in the $(t+1)$--st round.
	
	To see that (iii) holds, note that by (i) Maker's graph would be a spanning tree after round $n-1$, so that in every further round she cannot create any new threats.
	By (ii) it follows that $ t + 1 \leq n - 1$.
\end{proof}

The next lemma establishes that Breaker, when following \cref{str-mb-oc-con-br} but still losing, never claims any edge within $R_s$ or between $V_s^1$ and $V_s^2$ during round $s$.

\begin{obs}\label{Obs2}
	Breaker never executes parts (ii) (a) or (d) of \cref{str-mb-oc-con-br}.
\end{obs}

\begin{proof}
	If at some point in the game Breaker were to execute part (ii) (d) of his strategy, that is he selects an arbitrary unclaimed edge, then this also implies that $R_s$ is fully connected to $V_s$ in Breaker's graph, as he no longer was able to execute parts (ii) (b) and (ii) (c) of his strategy.
	It follows that Maker could only claim edges between $V_s^1$ and $V_s^2$, contradicting \cref{Obs1}~(i).

	If Breaker executes part (ii) (a) of the strategy once during round $s$, then he will in fact continue claiming edges between, without loss of generality, $V_s^1$ and $R_s$ during round $s$ until the two sets are fully connected.
	Note that he manages to do so before round $s$ is over.
	By \cref{Obs1}~(i), it follows that Maker will be forced to select an edge between $V_s^2$ and $R_s$ in her $(s+1)$--st move (since $R_s$ cannot be fully connected to both $V_s^1$ and $V_s^2$ by the previous paragraph), creating no new threats.
	Breaker therefore is able to continue executing part (ii) (a) of the strategy, always keeping $R_s$ and $V_s^1$ fully connected, until Maker has claimed a spanning tree, contradicting \cref{Obs1}~(iii).
\end{proof}

Even though \cref{obs:mb-oc-breaker} states that \emph{during round $s$} Breaker only ever selects edges between $R_s$ and $V_s$, an edge that was \emph{previously} claimed by Breaker between $R_{\tilde{s}}$ and $V_{\tilde{s}}$ for some $\tilde{s} < s$ can still end up as an edge between $V_{s}^1$ and $V_s^2$ through Maker's choices.
Estimating how many of Breaker's edges end up between the two parts of Maker's tree will be the central ingredient in the proof of \cref{thm-mb-oc-breaker}.
However, let us first establish that Breaker is succesful in his attempt to distribute his edges between $R_s$ and $V_s^1$ as well as $V_s^2$ in an even manner.

\begin{lemma} \label{lem-deg-reg}
	For any $1 \leq s \leq t + 1$, $0 \leq k \leq b$ and $ u,v \in R_s $ we have 
	\begin{enumerate}[(a)] \setlength{\itemsep}{0pt}
		\item $|\deg_k(u,V_s)-\deg_k(v,V_s)| \leq 2$ and  \label{lem-deg-reg:RVs}
		\item $|\deg_k(u,V_s^i)-\deg_k(v,V_s^i)| \leq 1$ for $ i \in \{1,2\}$.
\label{lem-deg-reg:RVsi}
	\end{enumerate}
\end{lemma}

\begin{proof}
  It suffices to prove part~(b) since part~(a) is an immediate consequence of it.
	We begin by observing that after Maker's $(s+1)$--st move $R_s$ gets reduced by one vertex by \cref{Obs1}, but none of the degrees of the remaining vertices change.
	Let us therefore consider what happens during Breaker's moves when he follows \cref{str-mb-oc-con-br}.
	We trivially note that, whenever there is a free edge between $ R_s $ and $ V_s^i $ for some $i \in \{1, 2\}$, there is a free edge between any vertex $ v \in R_s $ with minimal $ \deg_k(v,V_s^i)$ and $ V_s^i $.
	This basic observation combined with an elementary inductive argument implies part~(b) of the claim, since Breaker only ever claims an edge between $v$ and $V_s^i$ if $v$ minimises $ \deg_k(v,V_s^i)$ in \cref{str-mb-oc-con-br}.
%
	%
	%
\end{proof}

As promised, we are now able to show that Maker, assuming she wins, forces Breaker to have previously claimed almost all edges between $V_s^1$ and $V_s^2$ for any $1 \leq s \leq t$.
Let us refer to all edges between $V_s^1$ and $V_s^2$ not claimed by Breaker as edges that Breaker \emph{saved}.
Note that, by \cref{Obs2}, all edges which are saved at some point of the game remain saved throughout the game.

\begin{lemma}\label{lem-save-edges}
	Breaker saved at most $\left\lfloor (\e n^2 - n) / 2 \right\rfloor$ edges by the end of round $t$ if $n \geq 34$.
\end{lemma}

\begin{proof}
	By \cref{Obs1} (ii), there must exist some vertex $ v \in R_{t} $ such that at least $b+2$ of the edges between $v$ and $V_t$ are unclaimed, that is $ \deg(v,V_t) \leq t-b-1 $.
	By \cref{lem-deg-reg} (a), it follows that $\deg(u,V_t) \leq t-b+1$  for every $u \in R_t$.

	Assume now to the contrary that Breaker saved strictly more than $\left\lfloor (\e n^2 - n) / 2 \right\rfloor$ edges inside $ V_t $ at the end of round $t$. Since Maker has claimed exactly $t$ edges inside $V_t$ by \cref{Obs1} (i) and \cref{Obs2}, Breaker has claimed at most
	\begin{equation*}
		bt < \binom{t+1}{2} - t - \e n^2 /2 + n/2 + (n-t-1)(t-b+1) 
	\end{equation*}
	edges at the end of round $t$.
	From this we get the contradiction
	\begin{align*}
    b < \frac{-t^2 + 2 t (n - 2) + 3n - \e n^2 -2}{2(n-1)} &\leq \frac{(n-2)^2 + 3n - \e n^2 -2}{2(n-1)}  \\
    &= \frac{(n^2 - \e n^2) - (n  - 2)}{2(n-1)} \leq \left\lceil \frac{n-\e n}{2} \right\rceil
	\end{align*}
	where for the second inequality we have used the fact that the numerator is maximised when $t = n - 2$ and for the last inequality we are using that $\e = 0.06$ and $n \geq 34$.
\end{proof}

Combining \cref{Obs2} and \cref{lem-save-edges} will allow us to deduce good bounds on the number of edges between $ V_s $ and $ R_s $.
Let us write
\begin{equation}
	d_s = \frac{e(V_s,R_s)}{|R_s|} \quad \text{and} \quad d_s^i = \frac{e(V_s^i,R_s)}{|R_s|}
\end{equation}
for the average number of neighbours of a vertex $ v \in R_s $ in $ V_s $ and $V_s^i$ for $i \in \{1,2\}$.
We start by bounding $d_s^i$ with respect to $d_s$.

\begin{lemma} \label{lem-avgdeg-reg}
	For any $1 \leq s \leq t$ and $ i \in \{1,2\}$ we have 
  \begin{equation}\label{eq:avdeg-reg}
		\frac{d_s-1}2 \leq d_s^i \leq \frac{d_s+1}2.
	\end{equation}
\end{lemma}

\if10
\begin{proof}
	Let us call a state $(s,k)$ \emph{safe} if we have $\card{e_k(V_s^{1},R_s) - e_k(V_s^{2},R_s)} \leq 1$. Note that the initial state $(1,0)$ is safe.
	If the sate $(s,b)$ is safe, then after Maker's $s+1$--st move we might have, without loss of generality, that
	\begin{equation} \label{eq-edge-safe-maker}
		e_0(V_{s+1}^{1},R_{s+1}) = e_0(V_{s+1}^{2},R_{s+1}) - 2.
	\end{equation}
	If Maker claimed an edge between $V_{s}^{2}$ and some vertex $u \in R_s$, then Breaker will be able to follow part (ii) (b) of \cref{str-mb-oc-con-br} and claim an edge between $R_{s+1}$ and $V_{s+1}^{1}$, since $u \in V_{s+1}^{1}$ is completely disconnected from $R_{s+1}$ by \cref{Obs2}, which evens out the difference and returns to a safe state.
	Assume therefore that Maker claimed an edge between $V_{s}^{1}$ and some vertex $u \in R_s$.
	If Breaker is not be able to follow part (ii) (b) to even out the difference, then $V_{s+1}^{1}$ must be fully connected to $R_{s+1}$.
	Since Maker was able to claim that edge, we had $\deg (u,V_s^{1}) = \deg (v,V_s^{1}) - 1$ for all $v \in R_s \setminus \{v\}$.
	However, since $(s+1,0)$ is not safe, we must have had $\deg (u,V_s^{1}) = \deg(u,V_s^{2}) + 1$.
	By \cref{lem-deg-reg} (b) it follows that in fact $\deg (v,V_s^{1}) = \deg(v,V_s^{2}) + 1$ for all $v \in R_s$, so that $e(R_s,V_s^{1}) = e(R_s,V_s^{2}) + |R_s|$, a contradiction.
	It follows that at point $(s+1,1)$ we have returned to a safe state.
	
	Now suppose that the state $(s, k_0-1)$ is safe, but the subsequent state $(s, k_0)$ is not for some $1 \leq s \leq t$ and $1 \leq k_0 \leq b$.
	We show that Breaker will restore a safe state within round $s+1$.
	Since $(s,k_0)$ is not safe, all edges between $R_{s}$ and, without loss of generality, $V_{s}^1$ must have been claimed at state $(s,k_0-1)$.
	Since Breaker can claim at most $b$ edges between $R_{s}$ and $V_{s}^2$ during the remainder of round $s$, it follows that for any $k_0 \leq k \leq b$ we have
	\begin{align}\label{eq-edge-diff-s}
		|e_{k}(R_{s}, V_{s}^1) - e_{k}(R_{s}, V_{s}^2)| \leq b + 1.
	\end{align}
	By \cref{Obs1} $(i)$, Maker will now be forced to claim an edge between $R_s$ and $V_s^2$ in her $(s+1)$--st turn, which does not create any new threats. We also note that we had $\deg_{k_0-1}(v, V_s^1) = d_s^1$ for any $v \in R_s$ and that $\deg_{k_0-1}(v, V_s^2) = \deg_{k_0-1}(v, V_s^1) = d_s^1$ for all but at most one $v \in R_s$ by \cref{lem-deg-reg} (b) and the fact that state $(s,k_0-1)$ was safe.  If there exists $v_0 \in R_s$ with $\deg_{k_0-1}(v_0, V_s^2) = d_s^1-1$, Breaker will claim an edge between it and $V_s^2$ in move $k_0$, so that we may assume $\deg(v, V_s^2) \geq \deg(v, V_s^1)$ for all $v \in R_s$. We therefore have
	\begin{align}\label{eq-edge-diff-s1}
		|e_0(R_{s+1}, V_{s+1}^1) - e_0(R_{s+1}, V_{s+1}^2)| \leq b + 1.
	\end{align}
	Denote the edge that Maker claims in round $s+1$ by $vu$, where $u\in V_s^2$ and $v\in R_s$.
	At the beginning of Breaker's turn in round $s+1$ all vertices of $V_{s+1}^1$ except for $v$ are fully connected to $R_{s+1}$.
	By \cref{Obs2} we therefore must have
	\begin{equation} \label{eq-Rsb}
		|R_s| - 1 = |R_{s+1}| > b.
	\end{equation}
	Since there are $|R_{s+1}| > b$ unclaimed edges between $R_{s+1}$ and $V_{s+1}^1$, Breaker will even out the difference between $e(R_{s+1}, V_{s+1}^1)$ and $e(R_{s+1}, V_{s+1}^2)$ within round $s+1$ and return to a safe state, that is
	\begin{align}\label{eq-edge-diff-s2}
		|e_k(R_{s+1}, V_{s+1}^1) - e_k(R_{s+1}, V_{s+1}^2)| < b + 1.
	\end{align}
	for $1 \leq k < b$ and 
	\begin{align}\label{eq-edge-diff-s3}
		|e(R_{s+1}, V_{s+1}^1) - e(R_{s+1}, V_{s+1}^2)| \leq 1.
	\end{align}

	It remains to be observed that, by \cref{eq-Rsb}, whenever a state is safe or if \cref{eq-edge-safe-maker}, \cref{eq-edge-diff-s}, \cref{eq-edge-diff-s1}, \cref{eq-edge-diff-s2} or \cref{eq-edge-diff-s3} holds, then $\card{d_s^{1} - d_s^{2}} \leq 2$ and therefore in fact $(d_s-2)/2 \leq d_s^i \leq (d_s+2)/2$ for $i \in \{1,2\}$ as stated.
	This completes the proof.	
\end{proof}
\fi

\begin{proof}
Let us call a state $(s,k)$ \emph{safe} if there exists some $u_0 \in R_s$ such that
\begin{equation*}
	|\deg(u_0,V_s^{1}) - \deg(u_0,V_s^{2})| \leq 1 \quad \text{and} \quad \deg(u,V_s^{1}) = \deg(u,V_s^{2})
\end{equation*}
for any $u \in R_s \setminus \{u_0\}$.
We note that any safe state satisfies $\card{e_k(V_s^1,R_s)-e_k(V_s^2,R_s)} \leq 1$.
Clearly the initial state $(1,0)$ is safe and \cref{eq:avdeg-reg} holds if the state $(s,b)$ is safe. We also note that a state remains safe if Maker makes a move or if Breaker executes parts (i) or (ii) (b) of \cref{str-mb-oc-con-br}. We will show that, whenever some state becomes unsafe, Breaker will quickly return to a safe state without \cref{eq:avdeg-reg} ever being violated.

Suppose that $(s_0, k_0)$ is not safe for some $s_0 \leq t$ and $1 \leq k_0 \leq b$, but the previous state $(s_0, k_0-1)$ is safe.
By definition of \cref{str-mb-oc-con-br} and by \cref{Obs2}, all edges between $R_{s_0}$ and one part, say $V_{s_0}^1$, must be claimed by Breaker at state $(s_0,k_0-1)$.
It follows that, from his $k_0$--th move on, Breaker claims exactly $b-k_0+1$ edges incident to $R_{s_0}$ and $V_{s_0}^2$ in round $s_0$, that is
\begin{equation} \label{eq-edge-diff}
	2 \leq e(R_{s_0}, V_{s_0}^2) - e(R_{s_0}, V_{s_0}^1) \leq 1 + (b-k_0+1) \leq b+1.
\end{equation}
By \cref{Obs1} $(i)$, Maker must claim an edge incident to $R_{s_0}$ and $V_{s_0}^2$ in her $(s_0+1)$--st turn, therefore not creating any new threats of closing an odd cycle.
Since $\deg(u,V_{s_0}^1) \leq \deg(u,V_{s_0}^2)$ for any $u \in R_{s_0}$, it follows that after Maker's move we still have
\begin{equation*}
	2 \leq e_0(R_{s_0+1}, V_{s_0+1}^2) - e_0(R_{s_0+1}, V_{s_0+1}^1) \leq b+1.
\end{equation*}
Let the edge claimed by Maker in round $s_0+1$ be incident to $v_0 \in R_{s_0}$.
We know that all vertices of $V_{s_0+1}^1 = V_{s_0}^1 \cup \{v_0\}$ except for $v_0$ are fully connected to $R_{s_0+1}$ at the beginning of Breaker's turn in round $s_0+1$.
This means that all but exactly $|R_{s_0+1}|$ edges between $V_{s_0+1}^1$ and $R_{s_0+1}$ are claimed.
Since Breaker never executes part (ii) (a) of \cref{str-mb-oc-con-br} by \cref{Obs2}, we must have
\begin{equation} \label{eq:RS01}
	|R_{s_0+1}| \geq b+1.
\end{equation}
In particular, Breaker begins round $s_0 + 1 $ by claiming edges incident to $V_{s_0+1}^1$ and $R_{s_0+1}$ until
\begin{equation} \label{eq:returnedsafe}
	e_{k_1}(R_{s_0+1}, V_{s_0+1}^1) = e_{k_1}(R_{s_0+1}, V_{s_0+1}^2) - 1
\end{equation}
for some $0 \leq k_1 \leq b$.
By the way that ties are broken in part (ii) (b) of \cref{str-mb-oc-con-br}, it follows that $\deg_{k_1}(u,V_{s_0+1}^1) \leq \deg_{k_1}(u,V_{s_0+1}^2)$ for any $u \in R_{s_0+1}$ and hence \cref{eq:returnedsafe} in fact implies that $(s_0+1,k_1)$ is again a safe state.

Finally, since \cref{eq-edge-diff} and \cref{eq:RS01} imply that $|d_{s_0}^1 - d_{s_0}^2| \leq 1$, this completes the proof.
\end{proof}

Let us combine all of the previous results to describe the distribution of edges and vertices throughout the game.

\begin{lemma} \label{lem-avgdeg-size}
	For $1 \leq s \leq 2n/3 $, $s \leq t$ and $n \geq 5$ we have
	%
	%
	\begin{equation}
		\frac{s}{2} - 2\e n < d_s < \frac{s}{2} + \frac{\e n}2 + 4.
	\end{equation}
\end{lemma}

\begin{proof}
	\cref{lem-save-edges} implies the lower bound
	\begin{align*}
		e(V_s) \geq \binom{s + 1}{2} - s - \left\lfloor \frac{\e n^2 - n}{2} \right\rfloor \geq \frac{s^2 - s - \e n^2 + n}{2}
	\end{align*}
	so that by \cref{Obs2}
	\begin{align*}
		e(V_s,R_s) = bs - e(V_s) & \leq \frac{n-\e n + 1}{2} \, s - \frac{s^2 - s - \e n^2 + n}{2}. 
	\end{align*}
	Dividing by $|R_s| = n-s-1$, we get
	\begin{align*}
		d_s \leq \frac{s + \e n}{2} + \frac{3s - n +\e n}{2 \, (n-s-1)} < \frac{s + \e n}{2} + 4
	\end{align*}
	where the last inequality holds by our assumptions on the size of $s$ and $n$.
	On the other hand, we also have
	\begin{align*}
		e(V_s,R_s) \geq \frac{n - \e n}{2} \, s - \frac{s^2-s}{2} > \frac{(s - 4 \e n)(n-s-1)}{2}
	\end{align*}
	where the last inequality holds by our assumptions on the size of $s$. Dividing by $|R_s| = n-s-1$ gives us $d_s > (s - 4 \e n)/2$.
\end{proof}


\begin{cor} \label{cor-avgdeg-size}
	For $1 \leq s \leq 2n/3 $, $s \leq t$ and $n \geq 5$ we have
	\begin{equation}
		\frac{s}4 - \e n - \frac12 < d_s^i < \frac{s}4 + \frac{\e n}4 + \frac52 \quad \text{for } i \in \{1,2\}.
	\end{equation}
\end{cor}

Lastly, we will need two more technical lemmas before we can prove \cref{thm-mb-oc-breaker}.

\begin{lemma} \label{size_lemma}
	For every $1 \leq s \leq ( 2/3 -\sqrt{\e} )n$ there exists $ i \in \{1,2\} $ such that
	\begin{equation}
		|V_s^i| < d_s^i + 3 \sqrt{\e}n / 2
	\end{equation}
	for $n \geq  184$ .
\end{lemma}

\begin{proof}
	Suppose to the contrary that $ |V_{s_0}^i| \geq  d_{s_0}^i + 3 \sqrt{\e}n/2$ for both $i \in \{1,2\}$ and some fixed $1 \leq s_0 \leq ( 2/3 -\sqrt{e} )n$.
Using \cref{lem-deg-reg} (a) we get
	\begin{align*}
		d_{s+1} = \frac{e(V_{s+1},R_{s+1})}{|R_{s+1}|} &\leq \frac{e(V_{s},R_{s}) + b - (d_{s} - 2)}{|R_{s+1}|}  \\
		&= \frac{e(V_{s},R_{s}) - d_{s}}{|R_{s}|-1} + \frac{b+2}{|R_{s}|-1}  \\
		&= d_{s} + \frac{\lceil (n-\e n ) / 2 \rceil + 2}{n-s-2}  \\
		& \leq d_{s} + 3/2
	\end{align*}
  for every $1 \leq s \leq 2n/3$, where the last inequality holds since $\e = 0.06$ and $n \geq  184$. 
	By iterating this and using \cref{lem-deg-reg} (b) and \cref{lem-avgdeg-reg}, we get 
	\begin{align*}
		\deg(u,V_{s_0+j}^i) & \leq d_{s_0+j}^{i} +1 \leq \frac{d_{s_0+j}}{2} + \frac32 \\
		& \leq \frac{d_{s_0} + 3j/2}{2} + \frac32 \leq d_{s_0}^i + 3j/4 + 2
	\end{align*}
	for every $ u \in R_{s_0+j} \subset R_{s_0} $, $ j=0, \ldots , \lfloor \sqrt{\e} n \rfloor $ and $ i \in \{1,2\}$.
	By our assumption on the cardinality of $V_{s_0}^i$, Breaker therefore saves at least
	\begin{align*}
		|V_{s_0+j}^i| - \max_{u \in R_{s_0+j}} \deg(u,V_{s_0+j}^i) & \geq |V_{s_0}^i| - \big( d_{s_0}^i + 3j/4 + 2 \big) \\
		& \geq 3 \sqrt{\e}n/2s - 3j/4 - 2
	\end{align*}
	edges in of the rounds $ s_0+1, \dots, s_0 + \lfloor \sqrt{\e}n \rfloor $.
	In total Breaker therefore saves at least
	\begin{align*}
		\sum_{j=0}^{\lfloor \sqrt{\e} n \rfloor} \Big( 3 \sqrt{\e}n/2 - 3j/4 - 2 \Big) & = \lfloor \sqrt{\e} n +1 \rfloor (3 \sqrt{\e} n/2 - 2) - 3/4 \, \binom{\lfloor \sqrt{\e} n \rfloor + 1}2  \\
		& \geq 9 \e n^2/8  - 3 \sqrt{\e} n > \frac{\e n^2 -n}2
	\end{align*}
	edges throughout the game, contradicting \cref{lem-save-edges}.
\end{proof}

\begin{lemma}\label{lem-m-late-win}
	Maker does not win early, that is $t \geq 2n/3$ if $n \geq 5$.
\end{lemma}

\begin{proof}
	Without loss of generality, we assume that Maker claimed an edge between $V_t^1$ and some vertex $v \in R_t$ in the $(t+1)$--st round.
	By \cref{Obs1} (ii), this created at least $b +1$ threats and therefore $|V_t^2|-\deg(v,V_t^2) \geq b+1 $.
	By \cref{Obs1} (i) and \cref{lem-deg-reg} (b), we therefore have
	\begin{equation*}
		t+1 - d_t \geq t+1 - |V_t^1|-d_t^2 = |V_t^2|-d_t^2 \geq b.
	\end{equation*}
	If we assume that $t < 2n/3$, then by \cref{lem-avgdeg-size} we have
	\begin{align*}
		\frac{t + 2 + 4 \e n}{2} > t + 1 - d_t \geq b \geq \frac{(1-\e)n}{2},
	\end{align*}
	a contradiction.
\end{proof}

We are finally able to complete the proof of \cref{thm-mb-oc-breaker}.

\begin{proof}[Proof of \cref{thm-mb-oc-breaker}]
	Assume that $n \geq  184$ and consider the two rounds
	\begin{align*}
		s_1 = \left\lfloor \frac{(2/3 -\sqrt{\e}) n}2 \right\rfloor \quad \text{and} \quad s_2 = 2s_1.
	\end{align*}
	Note that $s_1,s_2 < t$ by \cref{lem-m-late-win}.
	By \cref{size_lemma} and \cref{cor-avgdeg-size} we now have, without loss of generality, that
	\begin{equation} \label{size_eqn1}
		|V_{s_1}^1| < d_{s_1}^1 + 3\sqrt{\e}n/2 < s_1/4 + (\e / 4 + 3\sqrt{\e}/2)n + 5/2.
	\end{equation}
	By \cref{Obs1} (i), it follows that
	\begin{equation}\label{size_eqn2}
		|V_{s_1}^2| = s_1 + 1 - |V_s^1| > 3s_1/4 - (\e / 4 + 3\sqrt{\e}/2)n - 5/2.
	\end{equation}
	Again by \cref{cor-avgdeg-size}, we also have
	\begin{equation}\label{size_eqn3}
		|V_{s_2}^1| \geq d_{s_2}^1 > s_2/4 - \e n - 1/2.
	\end{equation}
	Clearly \cref{size_eqn1} and \cref{size_eqn3} imply that Maker adds a vertex to $V_s^1$ during at least
	\begin{equation}
		(s_2-s_1)/4 - (3\sqrt{\e}/2 + 5\e / 4)n - 2
	\end{equation}
	of the rounds between rounds $s_1+1$ and $s_2$.
	By \cref{size_eqn2}, \cref{lem-deg-reg} (b) and \cref{cor-avgdeg-size} we know that Breaker saves at least
	\begin{align*}
		|V_{s_1}^2| - d_{s_1}^2 - 1 & > s_1/2 - (\e + 3\sqrt{\e})n/2 - 6
	\end{align*}
	edges in each of these rounds.
	Inserting $s_1$ and $s_2$, Breaker therefore saves at least
	\begin{align*}
		\left(\frac{2/3 -\sqrt{e}}8 - \frac{3\sqrt{\e}}2 - \frac{5\e}4 \right) \left( \frac{2/3 -\sqrt{e}}4- \frac{\e}2 - \frac{3\sqrt{\e}}2 \right) n^2 - o(n^2) > \e n^2 - o(n^2)
	\end{align*}
	edges in total, contradicting \cref{lem-save-edges} if $n$ is large enough.
\end{proof}

\section{A strategy for Client\texorpdfstring{ -- Proof of \cref{thm-client-occ}}{}} \label{sec-client-occ}

In this section we prove \cref{thm-client-occ} by presenting a strategy for Client in the connected Client-Water odd cycle game. We note that Client's graph, as long as she has not yet won the game, will be bipartite. If at any point there is an unclaimed edge inside either of the two parts of that bipartition, Waiter will be forced to eventually offer it, allowing Client to close an odd cycle and therefore winning the game. It follows that Waiter, whenever he offers an edge incident to a vertex which is not yet part of Client's graph, will either offer all unclaimed edges between that vertex and Client's graph or he must have previously claimed all edges between that vertex and one part of the bipartition. Client's strategy will be aimed at reducing the number of times the later of the two scenarios occurs.

Let $ G_C(s) = (V_s,E_s) $ denote Client's graph after her $s$--th turn and let $ R_s = [n] \setminus V_s $ be the set of vertices not touched by Client. As already mentioned, $ G_C(s) $ is bipartite as long as Client has not won the game and since $ G_C(s) $ is connected, there is a unique (up to labelling) bipartition $ V_s = V_s^1 \cup V_s^2 $, which we may choose in such a way that $ V_s^i \subset V_{s+1}^{i} $ holds for all $ s \geq 0 $ and $ i \in \{1,2\} $. Note again that our notion of Client's and Waiter's graph do not include isolated vertices.

Unless stated otherwise, we are always referring to Waiter's graph when talking about edges, degrees, etc.\ for the remainder of this section. Using this notation, let us introduce the following definition.

\begin{defi}
	Let $i \in \{1, 2\}$.
	A vertex $ v \in R_s $ is \emph{critical} with respect to $ V_s^i $, if every edge between it and $ V_s^i $ has already been claimed by Waiter by the end of round $s$.
	A part $ V_s^i $ is \emph{critical} if there exists a vertex $ v \in R_s $ which is critical with respect to $ V_s^i$.
\end{defi}

The following strategy for Client tries to avoid having many critical vertices. Parts (i) and (ii) of it merely imply that Client should win the game whenever she gets a chance to do so and therefore should be part of every rational strategy for Client.
Part (v) implies that Client chooses to forfeit rather than claiming an edge that would close an even cycle in her graph.
In particular, as long as Client has not won the game, her graph will be a tree.

\begin{str} \label{str-client-occ}
	Let $ W_s $ be the set of edges offered by Waiter in round $ s \geq 1$.
	\begin{enumerate}[(i)] \setlength{\itemsep}{0pt}
		\item If there is some edge in $W_s $ closing an odd cycle in Client's graph, she claims it.
		\item Otherwise, if there is some edge in $W_s $ so that after claiming it there would be an unclaimed edge in $V_s^1$ or $V_s^2$, she claims it.
		\item Otherwise, if there is some edge in $W_s $ which is incident to a non-critical part and to $R_s$, she claims it.
		\item Otherwise, if there is some edge in $W_s $ which is incident to $ R_s $, she claims it.
		\item Otherwise she forfeits.
	\end{enumerate}
\end{str}

From now on we assume that Waiter is given a bias of
\begin{equation}
	b = \left\lceil \frac{n}{2} \right\rceil - 2
\end{equation}
and that he wins the game despite Client following \cref{str-client-occ}. Let us state and prove three straight-forward lemmas before proving \cref{thm-client-occ}.

\begin{lemma} \label{rmk-criticalvx}
	There is at most one critical part at any time $s \geq 1$.
\end{lemma}

\begin{proof} 
	Since Client's graph is always a tree, she connects a new vertex $ v \in R_{s-1} $ to her graph in every round $ s \geq 1 $.
	Assume without loss of generality that the edge connects $v$ to a vertex in $V_{s-1}^1$.
	By the connected rules, no edge incident to $ v $ and any other vertex in $ R_{s-1} $ has been offered, or therefore claimed, at the end of round $s$, so that no edge between $v$ and $R_s$ will have been claimed at the end of round $s$. It follows that $V_s^2$ is not critical at the end of round $s$, implying the statement.
\end{proof}
	
\begin{lemma} \label{lem:noncritical}
	If Waiter offers an edge incident to some $ y \in R_s$ in round $ s + 1 \geq 1$ that is not critical to either $V_s^1$ or $V_s^2$, then Waiter in fact offers every unclaimed edge between $ y $ and $ V_s $ in round $s+1$.
\end{lemma}
	
\begin{proof}
	Suppose that Waiter offers the edge $ x_1 y $ for some, without loss of generality, $ x_1 \in V_s^1 $. We first show that Waiter offers all unclaimed edges between $y$ and $V_s^2$.

  Suppose for contradiction that there is an unclaimed edge $x_2 y$ for some $x_2 \in V_s^2$ which Waiter did not offer. Then, Client will choose $x_1 y$ (or an equivalent edge) by part (ii) of her strategy.
	Consequently, $x_2 y$ will be an unclaimed edge inside $ V_{s+1}^2 $ which Client, following part (i) of her strategy, will eventually pick and therefore close an odd cycle.
	This contradicts the assumption that Waiter wins the game and he therefore offers all unclaimed edges between $ y $ and $ V_s^2 $.
	
	Since $ y $ is not critical with respect to $ V_s^2 $, there is at least one unclaimed edges $ x_2 y $ for some $ x_2 \in V_s^2 $ which Waiter offers by the first part. Repeating the argument with $ x_2 $ and $ y $, we conclude that Waiter also offers all unclaimed edges between $ y $ and $ V_s^1 $.
\end{proof}
	
\begin{figure}
\centering
\begin{tikzpicture}
\edef \xdiam {3cm}
\edef \ydiam {\xdiam/2}
\edef \vxdiam {0.1cm}

\begin{scope}
\node[draw=black,fill=red!50, ellipse, outer sep = 0cm, inner sep = 0cm, minimum width = \xdiam, minimum height = \ydiam] (V1) at (0,0) {$ V_s^1 $};

\node[draw=black,fill=red!50, ellipse, outer sep = 0cm, inner sep = 0cm, minimum width = \xdiam, minimum height = \ydiam] (V2) at (0,2cm) {$ V_s^2 $};

\node[fill = black, circle, outer sep = 0cm, inner sep=0cm, minimum size = \vxdiam] (x) at ($(V1) + (0.3*\xdiam,0)$) {} (x) node[below] {$ x_1 $};

\node[fill = black, circle, outer sep = 0cm, inner sep=0cm, minimum size = \vxdiam] (y) at ($(V2) + (1.3*\xdiam/2,0)$) {} (y) node[right] {$ y $};

\foreach \x in {0,...,7}{  
	\pgfmathsetmacro\a{(-135 + 90/(2*7-1)) + 2*\x*90/(2*8-1)}
	\pgfmathsetmacro\b{135 - 2*\x*90/(2*8-1)}
	\coordinate (a\x) at ($(0,2cm) + (\a:1.5cm and 0.75cm)$);
	\coordinate (b\x) at ($(\b:1.5cm and 0.75cm)$);
}

\foreach \x in {0,1,...,7}
\draw[blue] (a\x) -- (b\x);

\foreach \x [count= \i] in {0,1,...,6}
\draw[blue] (a\x) -- (b\i);

\draw[dashed, blue] (x) -- (y);
\end{scope}

\begin{scope}[xshift=\xdiam * 2]
\node[draw=black,fill=red!50, ellipse, outer sep = 0cm, inner sep = 0cm, minimum width = \xdiam, minimum height = \ydiam] (V1) at (0,0) {$ V_s^1 $};

\node[draw=black,fill=red!50, ellipse, outer sep = 0cm, inner sep = 0cm, minimum width = \xdiam, minimum height = \ydiam] (V2) at (0,2cm) {$ V_s^2 $};

\node[fill = black, circle, outer sep = 0cm, inner sep=0cm, minimum size = \vxdiam] (x) at ($(V1) + (0.3*\xdiam,0)$) {} (x) node[below] {$ x_1 $};

\node[fill = black, circle, outer sep = 0cm, inner sep=0cm, minimum size = \vxdiam] (x') at ($(V2) + (0.3*\xdiam,0)$) {} (x') node[above] {$ x_2 $};

\node[fill = black, circle, outer sep = 0cm, inner sep=0cm, minimum size = \vxdiam] (y) at ($(V1) + (1.3*\xdiam/2,0)$) {} (y) node[right] {$ y $};

\foreach \x in {0,...,7}{  
	\pgfmathsetmacro\a{(-135 + 90/(2*7-1)) + 2*\x*90/(2*8-1)}
	\pgfmathsetmacro\b{135 - 2*\x*90/(2*8-1)}
	\coordinate (a\x) at ($(0,2cm) + (\a:1.5cm and 0.75cm)$);
	\coordinate (b\x) at ($(\b:1.5cm and 0.75cm)$);
}

\foreach \x in {0,1,...,7}
\draw[blue] (a\x) -- (b\x);

\foreach \x [count= \i] in {0,1,...,6}
\draw[blue] (a\x) -- (b\i);

\draw[dashed, blue] (x') -- (y);
\end{scope}

\end{tikzpicture}
\caption{The situation in \cref{lem:noncritical}.
Waiter's edges are in red and Client's in blue.}
\label{fig-cw-oc-waiter}
\end{figure}
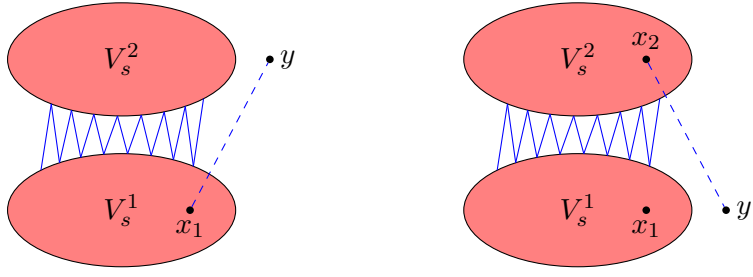
	
\begin{lemma}\label{lem:critical}
	If $ v \in R_s $ is critical to either $V_s^1$ or $V_s^2$, then there is exactly one unclaimed edge between $ v $ and $ V_s $.
\end{lemma}
	
\begin{proof}
	Assume that the statement does not hold and let $ s_0 \geq 1$ be the first time such that there is a vertex $ v \in R_{s_0} $ critical to, without loss of generality, $V_s^1$ and there are at least two unclaimed edges between $ v $ and $ V_s^2 $.
	
	If $ v $ was not critical with respect to $ V_{s_0-1}^1 $, then in round $s_0$ Waiter offered at least one edge incident to $ v $ and therefore all edges between $ v $ and $ V_{s_0-1} $ by \cref{lem:noncritical}. This contradicts the assumption that $v$ is not critical to $V_{s_0}^2$.
				
	Hence $ v $ was critical with respect to $ V_{s_0-1}^1$ and by our choice of $ s_0 $ there was exactly one unclaimed edge between $ v $ and $ V_{s_0-1}^2 $, see \cref{fig-cw-oc-proof}.
	Client therefore must have claimed an edge $ xy $ with $ x \in V_{s_0-1}^1 $ and $ y \in R_{s_0-1} $ in round $ s_0 $ to enlarge $ V_{s_0-1}^2 $.
	But $ y $ was not critical with respect to $ V_{s_0-1}^2 $ by \cref{rmk-criticalvx} and hence Waiter offered at least one edge incident to $ y $ and $ V_{s_0-1}^2 $ by \cref{lem:noncritical}.
	But following part (iii) of \cref{str-client-occ}, Client would have chosen this edge instead, a contradiction.
\end{proof}
	
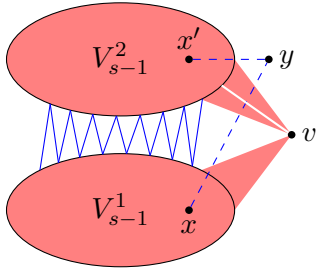
\begin{figure}
\centering
\begin{tikzpicture}
\edef \xdiam {3cm}
\edef \ydiam {\xdiam/2}
\edef \vxdiam {0.1cm}

\begin{scope}
\node[draw=black,fill=red!50, ellipse, outer sep = 0cm, inner sep = 0cm, minimum width = \xdiam, minimum height = \ydiam] (V1) at (0,0) {$ V_{s-1}^1 $};

\node[draw=black,fill=red!50, ellipse, outer sep = 0cm, inner sep = 0cm, minimum width = \xdiam, minimum height = \ydiam] (V2) at (0,2cm) {$ V_{s-1}^2 $};

\node[fill = black, circle, outer sep = 0cm, inner sep=0cm, minimum size = \vxdiam] (v) at ($(V1) + (1.5*\xdiam/2, 1cm)$) {} (v) node[right] {$ v $};

\node[fill = black, circle, outer sep = 0cm, inner sep=0cm, minimum size = \vxdiam] (x) at ($(V1) + (0.3*\xdiam,0)$) {} (x) node[below] {$ x $};

\node[fill = black, circle, outer sep = 0cm, inner sep=0cm, minimum size = \vxdiam] (x') at ($(V2) + (0.3*\xdiam,0)$) {} (x') node[above] {$ x' $};

\node[fill = black, circle, outer sep = 0cm, inner sep=0cm, minimum size = \vxdiam] (y) at ($(V2) + (1.3*\xdiam/2,0)$) {} (y) node[right] {$ y $};

\foreach \x in {0,...,7}{  
	\pgfmathsetmacro\a{(-135 + 90/(2*7-1)) + 2*\x*90/(2*8-1)}
	\pgfmathsetmacro\b{135 - 2*\x*90/(2*8-1)}
	\coordinate (a\x) at ($(0,2cm) + (\a:1.5cm and 0.75cm)$);
	\coordinate (b\x) at ($(\b:1.5cm and 0.75cm)$);
}

\foreach \x in {0,1,...,7}
\draw[blue] (a\x) -- (b\x);

\foreach \x [count= \i] in {0,1,...,6}
\draw[blue] (a\x) -- (b\i);

\draw[dashed, blue] (x) -- (y);
\draw[dashed, blue] (x') -- (y);

\begin{pgfonlayer}{background}
\fill[red!50] ($(V1) + (1.5*\xdiam/2, 1cm)$) -- (V2.south east) -- (V2.east) -- cycle;
\fill[red!50] ($(V1) + (1.5*\xdiam/2, 1cm)$) -- (V1.north east) -- (V1.east) -- cycle;
\end{pgfonlayer}

\draw[white, thick] (v) -- ($(V2) + (-27.5:1.5cm and 0.75cm)$);

\end{scope}
\end{tikzpicture}
\caption{The situation in \cref{lem:critical}.
Waiter's edges are in red and Client's edges are in blue.
The white edge denotes an unclaimed edge and the dashed edges are two of the edges which are offered to Client.}
\label{fig-cw-oc-proof}
\end{figure}
	
Equipped with these lemmas, it is now easy to prove \cref{thm-client-occ}.

\begin{proof}[Proof of \cref{thm-client-occ}]
	By \cref{lem:noncritical,lem:critical}, whenever Waiter offers one edge incident to some $ v \in R_s $, he in fact offers all edges between $ v $ and $ V_s $ .
	It follows that there will never be an unclaimed edge inside $ V_s $ and therefore Client does not give up the game by following part (v) of her strategy.
	Hence the game ends after some
	\begin{equation*}
		t \geq \frac{\binom{n}{2}}{b+1}  = \frac{n(n-1)}{2\lceil n/2 \rceil - 2} \geq \frac{n(n-1)}{ n - 1} = n
	\end{equation*}
	rounds with all edges having been claimed.
	Since $ G_C(t) $ is a tree, this implies 
	\begin{equation*}
		|V_t| = |E_t| + 1 = t+1 \geq n+1,
	\end{equation*}
	a contradiction.
\end{proof}

\section{Concluding remarks and open questions} \label{sec-remarks}

Two other variants of the odd cycle game that could be studied are Avoider-Enforcer and Waiter-Client games (not to be confused with the Client-Waiter games studied in this paper).
In the biased Avoider-Enforcer odd cycle game, Enforcer wants to force Avoider to claim all edges of an odd cycle, whereas the appropriately named Avoider tries to avoid just that.
Since there are similar issues regarding monotonicity of the bias as in the Client-Water variant, there is a \emph{monotone} version of this game, where Avoider claims \emph{at least} one and Enforcer claims \emph{at least} $b$ edges in each round, as well as a \emph{strict} version, where Avoider and Enforcer respectively claim \emph{exactly} one and $b$ edges.
For more details on Avoider-Enforcer games see~\cite{TicTacToe,HKS07,AE-rules,Tiborbook}.

Hefetz, Krivelevich, Stojaković and Szabó~\cite{HKSS2009b} considered the unbiased Avoider-Enforcer odd cycle game, that is $b = 1$, and proved that Enforcer wins rather fast.
Clemens, Ehrenmüller, Person and Tran~\cite{AvoiderAcyclic} proved that for a bias satisfying $b \geq 200n\ln n$, Avoider can ensure that his graph is a forest for every but the last round of the game.
Hefetz, Krivelevich, Stojaković and Szabó~\cite{HKSScolorability} considered various Avoider-Enforcer games, among them the strict version of the biased odd cycle game.
In the strict version, there is an \emph{upper} threshold bias $b^+_{ae}(\mathcal{OC}_n)$ and a \emph{lower} threshold bias $b^-_{ae}(\mathcal{OC}_n)$ with the property that Avoider wins if $b \geq b^+_{ae}(\mathcal{OC}_n)$ and Enforcer wins if $b < b^-_{ae}(\mathcal{OC}_n)$.
The authors in~\cite{HKSScolorability} showed that $cn \leq b^-_{ae}(\mathcal{OC}_n) \leq b^+_{ae}(\mathcal{OC}_n) \leq n^{3/2}$ for some constant $c>0$.

\begin{question}
	What is the threshold bias for the Avoider-Enforcer odd cycle games?
\end{question}

Waiter-Client games are similar to Client-Waiter games, except now Client's goal is to avoid claiming all edges of a winning set.
Bednarska-Bzd\c{e}ga, Hefetz, Krivelevich, and Łuczak considered Waiter-Client games in \cite{WC-CW-probint}.
A special case of their result, regarding the Waiter-Client odd cycle game, states that $n - 4\lceil n^{3/4} \rceil + 1 \leq b_{wc}(\mathcal{OC}_n) \leq 1.00502n$ and they further conjectured that the lower bound is asymptotically correct.


\bigskip

One goal of \cref{thm-mb-oc-breaker} was to show that, if the answer to \cref{qu-mb-oc} is \emph{``Yes.''}, then we have separated the regular threshold from the connected one.
Motivated by this, we believe that it is of interested to determine the connected thresholds of various other games.
One classic game is the connectivity game, for which Gebauer and Szab\' o~\cite{ConnGame} showed that the threshold bias for the Maker-Breaker variant is asymptotically equal to $n/\ln n$.
Another classic game is the Hamiltonicity game, for which Krivelevich~\cite{HamCyc} showed that the threshold bias for the Maker-Breaker game is $(1+o(1))n/ \ln n$.
However, it is easy to see that, with a bias of $b=2$, Breaker can isolate a vertex when playing against connected Maker.
This means that for the Hamiltonicity and the connectivity games the regular biases differ quite drastically from the connected ones.
This is true as well for other games where Maker's goal is to occupy a spanning subgraph of $K_n$.

One can ask whether we observe the same phenomenon for other Maker-Breaker games.
For example, in the Maker-Breaker $H$-game, Maker's goal is to fully claim a copy of a fixed graph $H$.
Bednarska and \L uczak~\cite{BedLuczRandStr} showed that the threshold bias for the Maker-Breaker $H$-game is $\Theta (n^{1/m_2(H)})$ where $m_2(H) = \max_{H'\subseteq H} (e(H')-1)/(v(H')-2)$.
Kusch, Rué, Spiegel and Szabó~\cite{RKSS18} generalised their results to a large class of games, including van der Waerden games introduced by Beck in~\cite{Beck81}.
Since in these proofs Maker's strategy is not connected, it could be of interest to determine whether the threshold bias of these games is equal to the regular one.
\begin{question}
	What is the threshold bias for the connected Maker-Breaker $H$-game?
\end{question}

Lastly, let us note that, since a graph is $2$-colourable if and only if is does not contain odd cycles, one can view the odd cycle game as the non-$2$-colourability game.
Hence, a natural generalisation of the odd cycle game would be to study the non-$k$-colourability game for integers $k \geq 3$.
It was proved in \cite{HKSScolorability} that the threshold bias for the Maker-Breaker non-$k$-colourability game satisfies $s_1 n \leq b_{mb}(\mathcal{NC}_n^k) \leq s_2 n$, where $\mathcal{NC}_n^k = \{E(C) : C \text{ subgraph of }K_n \text{ that is not } k \text{-colourable} \}$ and $s_1 = s_1(k)$ and $s_2 = s_2(k)$ are constants depending only on $k$.
It would be interesting to determine whether the threshold biases of the connected and the regular non-$k$-colourability game are equal or not.

\begin{question}
	Do we have $b_{mb}(\mathcal{NC}_n^k) \sim b_{mb}^c(\mathcal{NC}_n^k)$?
\end{question}

\medskip
\noindent \textbf{Acknowledgements.} The research on this project was initiated and developed during joint research workshops of Tel Aviv University and the Freie Universit\"at Berlin on Positional Games in 2016 and on Graph and on Hypergraph Colouring Problems in 2018. We would like to thank the German-Israeli Foundation (GIF grant number G-1347-304.6/2016) and both institutions for their support.
Furthermore, we would like to thank Michael Krivelevich and Gal Kronenberg for fruitful discussions.
	
\bibliographystyle{abbrv}
\bibliography{./input/bib}

~\\
~\\
\noindent
\begin{tabular}{cc}
	\begin{minipage}[t]{\linewidth}
		\textbf{Jan Corsten}\\
		The London School of Economics and Political Science, Department of Mathematics, London WC2A 2AE, UK. Supported by an LSE studentship.
		
		E-mail: \href{mailto:j.corsten@lse.ac.uk}{j.corsten@lse.ac.uk}.
	\end{minipage}
	\\~\\
	\begin{minipage}[t]{\linewidth}
		\textbf{Adva Mond}\\
		Tel Aviv University,
		School of Mathematical Sciences, 
		Israel.
		
		E-mail: \href{mailto:advamond@mail.tau.ac.il}{advamond@mail.tau.ac.il}.
	\end{minipage}
	\\~\\
	\begin{minipage}[t]{\linewidth}
		\textbf{Alexey Pokrovskiy}\\ 
		Birkbeck College,
		University of London,
		United Kingdom.
		
		Email: \href{mailto:dr.alexey.pokrovskiy@gmail.com}{dr.alexey.pokrovskiy@gmail.com}.
	\end{minipage}
	\\~\\
	\begin{minipage}[t]{\linewidth}
		\textbf{Christoph Spiegel}\\
		Universitat Polit\`ecnica de Catalunya and Barcelona Graduate School of Mathematics, Department of Mathematics, Edificio Omega, 08034 Barcelona, Spain. 
		Supported by a Berlin Mathematical School Scholarship and by the Spanish Ministerio de Econom\'ia y Competitividad FPI grant under the project MTM2014-54745-P and the Mar\'ia de Maetzu research grant MDM-2014-0445.
		
		E-mail: \href{mailto:christoph.spiegel@upc.edu}{christoph.spiegel@upc.edu}.
	\end{minipage}
	\\~\\
	\begin{minipage}[t]{\linewidth}
		\textbf{Tibor Szab\'o}\\ 
		Institute of Mathematics, FU Berlin, 14195 Berlin, Germany. Research supported in part by GIF grant No. G-1347-304.6/2016.
		
		E-mail: \href{mailto:szabo@math.fu-berlin.de}{szabo@math.fu-berlin.de}.
	\end{minipage}
\end{tabular}
	
\end{document}